\documentclass[draft]{amsart}
\usepackage{amssymb,graphicx,enumerate,amsthm}
\usepackage{multirow}
\usepackage{xcolor}

\usepackage[margin=1in, marginparwidth=2cm]{geometry}

\numberwithin{equation}{section}
\newtheorem{theorem}{Theorem}[section]

\newtheorem{corollary}[theorem]{Corollary}

\newtheorem{conjecture}[theorem]{Conjecture}

\begin{document}

\makeatletter
\def\imod#1{\allowbreak\mkern10mu({\operator@font mod}\,\,#1)}
\makeatother

\author[George E. Andrews]{George E. Andrews}
\address{[G.E.A.] The Pennsylvania State University, Department of Mathematics, University Park, PA 16802}
\email{gea1@psu.edu}

\author[Ali K. Uncu]{Ali K. Uncu}
\address{[A.K.U.] Johann Radon Institute for Computational and Applied Mathematics, Austrian Academy of Science, Altenbergerstraße 69, A-4040 Linz, Austria}
\email{akuncu@ricam.oeaw.ac.at}
\address{[A.K.U.] University of Bath, Faculty of Science, Department of Computer Science, Bath, BA2\,7AY, UK}
\email{aku21@bath.ac.uk}

\title[\scalebox{.9}{Sequences in Overpartitions}]{Sequences in Overpartitions}
     
\begin{abstract} This paper is devoted to the study of sequences in overpartitions and their relation to 2-color partitions. An extensive study of a general class of double series is required to achieve these ends.
\end{abstract}
   
\thanks{Research of the first author is partially supported by the Simons foundation grant number 633284. Research of the second author is partly supported by EPSRC grant number EP/T015713/1 and partly by FWF grant P-34501N}   
   
  

\date{\today}
   
   
\maketitle

\section{Introduction}

The topic of sequences in partitions goes back to the work of Sylvester \cite{Sylvester} and MacMahon \cite{MacMahon}. In 2004, Holroyd, Liggett and Romik found the concept central to a problem in bootstrap percolation \cite{HLR}. Subsequently there have been a number of papers on this topic by Andrews \cite{A1, A2} and Bringmann et. al. \cite{BMN}. The topic was also considered in the unpublished Ph.D. thesis of Hirschhorn \cite{HirschhornThesis} and by many others.

This paper has its genesis in two identities discovered empirically by the second author. 

\begin{theorem}\label{thm:main_theorem}
\begin{equation}\label{eq:1mod3eqn} \sum_{m,n\geq 0} \frac{(-1)^n q^{\frac{3n(3n+1)}{2} + m^2+ 3mn}}{(q;q)_m(q^3;q^3)_n} = \frac{1}{(q;q^3)_\infty}.\end{equation}
\end{theorem}

The second identity is still unproven.

\begin{conjecture}\label{conj:main_conjecture}
\begin{equation}\label{eq:2_3mod6eqn} \sum_{m,n\geq 0} \frac{(-1)^n q^{\frac{3n(3n+1)}{2} + m^2+ 3mn +m + n}}{(q;q)_m(q^3;q^3)_n} = \frac{1}{(q^2,q^3;q^6)_\infty},\end{equation}
 where \[(A;q)_n := \prod_{j=0}^{n-1} (1-Aq^j),\text{  and  }(A_1,A_2,\dots,A_r;q) := \prod_{i=1}^r(A_i;q)_n.\]
\end{conjecture}

In light of the fact that the double series in \eqref{eq:1mod3eqn} and \eqref{eq:2_3mod6eqn} are clearly cousins of many of the generating functions alluded to in the first paragraph, it is not surprising that these series are special cases of more general series related to the theory of partitions.

In this paper, we shall study the following series and their applications to overpartitions:
\begin{equation}\label{eq:F_ikx_def}
F(i,k;x) = \sum_{m,n\geq 0} \frac{(-1)^n q^{{(2k+1)n+1\choose 2} + m^2 + (2k+1)mn + i(m+n)} x^{m+(2k+1)n}}{(q;q)_m (q^{2k+1};q^{2k+1})_n}.
\end{equation}

We note that $F(0,1;1)$ is the series in Theorem~\ref{thm:main_theorem} and $F(1,1;1)$ is the series in the Conjecture~\ref{conj:main_conjecture}.

Indeed, our focus will, for the most part, be on the polynomial refinements of \eqref{eq:1mod3eqn} and \eqref{eq:2_3mod6eqn}. Namely, for $k,j,N\geq 0$, 
\begin{align*}
F_N(i,j,k;x,q)&=F_N(i,j,k;x)=F_N(i,j,k) \\ \nonumber
&=\left\lbrace\begin{array}{ll}
\begin{array}{l}\displaystyle\sum_{m,n\geq 0} (-1)^n q^{{(2k+1)n+1\choose 2} + m^2 + (2k+1)mn + i(m+n)} x^{m+(2k+1)n} \\ \displaystyle \hspace{3cm} \times{N-(2k+1)n-m+j\brack m}_q {N-2kn -m \brack n}_{q^{2k+1}}\end{array},& \text{if  }N\geq 0, \\ \\
0, & \text{if  }N<0,\\
\end{array} \right.
\end{align*}
where
\[{A\brack B}_q := \left\lbrace \begin{array}{ll}0 & \text{if  } B>A\text{ or }B<0, \\
\dfrac{(q;q)_{A}}{(q;q)_B(q;q)_{A-B}},&\text{otherwise.} \end{array} \right. \]

Note $\lim_{N\rightarrow\infty} F_N(0,j,1;1)$ is the series in the Theorem~\ref{thm:main_theorem} and $\lim_{N\rightarrow\infty} F_N(1,j,1;1)$ is the series in the Conjecture~\ref{conj:main_conjecture}.

As we will show in Section~\label{sec:Overpartitions} \[\lim_{N\rightarrow\infty} \frac{F_N(i,0,k;1)}{(q;q)_\infty}\]
is the generating function for a general class of overpartitions. 

It should be noted that Theorem~\ref{thm:main_theorem} has a direct consequence.

\begin{corollary}\label{cor:Overpartition_corollary} The number of overpartitions of $N$ that do not contain instances of $\overline{j}+(j-1)+\overline{(j-2)}$ and $\overline{j}+\overline{(j-1)}$ equals the number of partitions of $N$ into red and green parts with each green part $\equiv 1 $ (mod $3$).
\end{corollary}

For example, when $N=4$, the 13 overpartitions in the first class are 
\begin{center}$4,\ \overline{4},\ 3+1,\ \overline{3}+1,\ 3+\overline{1},\ \overline{3}+\overline{1},\ 2+2,\ \overline{2}+2,\ 2+1+1,\ $\\$ \overline{2}+1+1,\ 2+\overline{1}+1,\ 1+1+1+1,\ \overline{1}+1+1+1,$\end{center}
and the 13 colored partitions in the second class are 
\begin{center}$ 4_r,\  4_g,\  3_r+1_r,\  3_r+1_g,\ 2_r+2_r,\  2_r+1_r+1_r,\ 1_r+1_g+1_r,\ $\\$ 2_r+1_g+1_g ,\  1_r+1_r+1_r+1_r,\ 1_g+1_r+1_r+1_r,\ $\\$1_g+1_g+1_r+1_r,\ 1_g+1_g+1_g+1_r,\ 1_g+1_g+1_g+1_g.$ \end{center}

Section~\ref{sec:FN_recs} will be devoted to determining the recurrences and $q$-difference equations satisfied by $F_N(i,j,k;x)$ and Section~\ref{sec:F_k0_reduction} will consider identities arising when $k=0$.

Section~\ref{sec:Main_thm} will provide the proof of Theorem~\ref{thm:main_theorem}. Section~\ref{sec:Overpartitions} will consider overpartition applications. Section~\ref{sec:ContFrac} is devoted to some continued fraction expansions related to Theorem~\ref{thm:main_theorem} and their implications.

\section{Recurrence and $q$-Difference Equations for $F_N(i,k;x)$}\label{sec:FN_recs}

In the following, we shall use the standard notations,

\begin{align*}
(A;q)_n&:=\prod_{j=0}^{n-1}(1-Aq^i),\hspace{.5cm}(A;q)_\infty :=\prod_{j=0}^{\infty}(1-Aq^i),\\
\intertext{and}
{A \brack B}_q &:= \left\lbrace \begin{array}{ll}0 & \text{if  } B>A\text{ or }B<0, \\
\dfrac{(q;q)_{A}}{(q;q)_B(q;q)_{A-B}},&\text{otherwise.} \end{array} \right.
\end{align*}

\begin{theorem}\label{thm:F_Nijkx_rec}
\begin{equation}
\label{eq:F_Nijkx_rec} F_N(i,j,k;x) = F_{N-1}(i,j,k;x) + x q^{N+j+k-1}  F_{N-2}(i,j,k;x) - x^{2k+1} q^{(2k+1)(N-k)+i}  F_{N-(2k+1)}(i,j,k;x).
\end{equation}
\end{theorem}

\begin{proof}
\begin{align}\label{eq:F_Nijkx_rec_openly_written_in_proof}
&F_N(i,j,k;x) - F_{N-1}(i,j,k;x) - x q^{N+j+k-1}  F_{N-2}(i,j,k;x) + x^{2k+1} q^{(2k+1)(N-k)+i}  F_{N-(2k+1)}(i,j,k;x)\\
\nonumber&\hspace{.5cm}= \sum_{m,n\geq 0} (-1)^n q^{{(2k+1)n+1\choose 2} + m^2 + (2k+1)mn + i(m+n)} x^{m+(2k+1)n} \times \\ 
\nonumber&\left\lbrace {N-(2k+1)n-m+j\brack m}_q {N-2kn-m\brack}_{q^{2k+1}}\right.-{N-1-(2k+1)n-m+j\brack m}_q{N-1-2kn-m\brack n}_{q^{2k+1}}\\ 
\nonumber&\hspace{2cm}-q^{N-(2k+1)n-2m+j} {N-1-(2k+1)n-m+j\brack m-1}_q {N-2kn -m-1\brack n}_{q^{2k+1}}\\ 
\nonumber&\hspace{2cm} \left. +q^{(2k+1)(N-(2k+1)n-m)} {N-(2k+1)n-m+j\brack m}_q {N-2kn -m-1\brack n-1}_{q^{2k+1}} \right\rbrace
\end{align}
Let us combine the first and the last term in the braces in \eqref{eq:F_Nijkx_rec_openly_written_in_proof}. This yields
\begin{align}\label{eq:F_Nijkx_rec_first_last_term_combination}
{N-(2k+1)n-m+j \brack m}_q &\left\lbrace {N-2kn-m \brack n}_{q^{2k+1}} - q^{(2k+1))(N-(2k+1)n-m)}{N-2kn-m-1\brack n-1}_{q^{2k+1}} \right\rbrace\\
&\nonumber={N-(2k+1)n-m+j \brack m}_q{N-2kn-m-1\brack n}_{q^{2k+1}}
\end{align} by \cite[(3,3,3), p.35]{Theory_of_Partitions}.

Next we combine the second and the third terms in the braces of \eqref{eq:F_Nijkx_rec_openly_written_in_proof}. This yields
\begin{align}\label{eq:F_Nijkx_rec_decond_third_term_combination}
-{N-2kn-m-1 \brack n}_{q^{2k+1}} &\left\lbrace {N-(2k+1)n-m+j \brack m}_q - q^{N-(2k+1)n-2m+j}{N-(2k+1)n-m+j-1\brack m-1}_q \right\rbrace\\
&\nonumber=-{N-(2k+1)n-m+j\brack m}_q{N-2kn-m-1 \brack n}_{q^{2k+1}}
\end{align} by \cite[(3,3,3), p.35]{Theory_of_Partitions}.

Thus we see that combining the second and the third terms \eqref{eq:F_Nijkx_rec_decond_third_term_combination} is the negative of the term we get by combining the first and the last terms \eqref{eq:F_Nijkx_rec_first_last_term_combination} of \eqref{eq:F_Nijkx_rec_openly_written_in_proof}. So the expression inside the braces of \eqref{eq:F_Nijkx_rec_openly_written_in_proof} vanishes and the theorem is proven.
\end{proof}

\begin{corollary}\label{cor:F_N011x_rec} For $N\geq 1$,
\begin{equation}
\label{eq:F_N011x_rec} F_N(0,1,1;x) = (1+xq^N) F_{N-1}(0,1,1;x) -x^2q^{2N-1} F_{N-2} (0,1,1;x).
\end{equation}
\end{corollary}

\begin{proof}
First we not the initial conditions 
\begin{align*}
F_{-1}(0,1,1;x)&=0,\\
F_0(0,1,1;x)&=1,\\
F_1(0,1,1;x)&=1+xq,\\
F_2(0,1,1;x)&=1+(q + q^2)x,\intertext{and}
F_3(0,1,1;x)&=1+(q+q^2+q^3)x+x^2 q^4-x^3 q^6.\\
\end{align*}
Thus we check directly that \eqref{eq:F_N011x_rec} is true for $1\leq N\leq 3$. 

If $N>3$, we know from Theorem~\ref{thm:F_Nijkx_rec} that \begin{align*}
0&= F_N(0,1,1;x)-F_{N-1}(0,1,1;x) -xq^N F_{N-2}(0,1,1;x) + x^3 q^{3N-3} F_{N-3}(0,1,1;x)\\
&=(F_N(0,1,1;x) - (1+xq^N) F_{N-1}(0,1,1;x) +x^2q^{2N-1} F_{N-2} (0,1,1;x))\\
&\hspace{1cm}+ xq^N (F_{N-1}(0,1,1;x) - (1+xq^{N-1}) F_{N-2}(0,1,1;x) +x^2q^{2N-3} F_{N-3} (0,1,1;x))
\end{align*}
Let \[S(N) = F_N(0,1,1;x)- (1+xq^N) F_{N-1}(0,1,1;x) +x^2q^{2N-1} F_{N-2} (0,1,1;x),\] then we have just shown that \begin{equation}\label{eq:S_N_rec}S(N) + xq^N S(N-1)=0.\end{equation} But we began with the observation that for $1\leq N\leq 3$, $S(N)=0$. Hence, \eqref{eq:S_N_rec} and these initial conditions prove that $S(N)\equiv 0$ for all $N\geq 1$.
\end{proof}

\begin{theorem}\label{thm:F_Nijkx_three_term_j_rec}
\begin{equation}
\label{eq:F_Nijkx_three_term_j_rec} F_N(i,j,k;x) = F_N(i,j-1,k;x) + x q^{N+i+j-1} F_{N-1} (i,j-1,k;x).
\end{equation}
\end{theorem}
\begin{proof}
\begin{align*}
F_{N}(i,j,k;x)&-F_N(i,j-1,k;x)\\
&= \sum_{m,n\geq 0} q^{ {(2k+1)n+1\choose 2} +m^2 +(2k+1)mn + i(m+n) } x^{m+(2k+1)n} {N-2kn-m\brack n}_{q^{2k+1}}\\
&\hspace{1cm}\times\left\lbrace {N-(2k+1)n-m+j\brack m}_q - {N-(2k+1)n-m+j-1\brack m}_q\right\rbrace
\intertext{by \cite[(3,3,3), p.35]{Theory_of_Partitions}}
&= \sum_{m,n\geq 0} q^{ {(2k+1)n+1\choose 2} +m^2 +(2k+1)mn + i(m+n) } x^{m+(2k+1)n} {N-2kn-m\brack n}_{q^{2k+1}}\\
&\hspace{1cm}\times\left\lbrace q^{N-(2k+1)n-2m+j} { N-(2k+1)n-m+j-1 \brack m-1}_q\right\rbrace 
\intertext{Now, by shifting the summation variable $m\mapsto m+1$ and rewriting the terms, we get}
&= \sum_{m,n\geq 0} q^{ {(2k+1)n+1\choose 2} +(m+1)^2 +(2k+1)(m+1)n + i(m+1+n) } x^{m+1+(2k+1)n} {N-2kn-m-1\brack n}_{q^{2k+1}}\\
&\hspace{1cm}\times\left\lbrace q^{N-(2k+1)n-2(m+1)+j} { N-(2k+1)n-m+j-2 \brack m-1}_q\right\rbrace \\
&= x q^{N+i+j-1} F_{N-1}(i,j-1,k;x).
\end{align*}
\end{proof}

\begin{theorem}\label{thm:F_N_ijkx_matrix_rec} 
\begin{equation}\label{eq:F_N_ijkx_matrix_rec}
F_N(i,0,k;x) - F_{N-1}(i,0,k;xq) - xq F_{N-2}(i,0,k;xq^2) + x^{2k+1}q^{{2k+2\choose 2}+i} F_{N-(2k+1)}(i,0,k;xq^{2k+1})=0.
\end{equation}
\end{theorem}

\begin{proof}
This result follows easily from the determinant representation of $F_N(i,0,k;x)$:
\begin{align*}F_N(i,0,k;x)&= \\ &\hspace{-1.5cm}\left|\begin{array}{ccccccccccc}
1 & xq & x^2q^3 & 0 & 0 & \dots & 0 & -x^{2k+1} q^{{2k+2\choose 2}+i} & 0 &\dots & 0 \\
-1 & 1 & x q^2 & x^2 q^5 & 0 & \dots & & & & & \vdots\\
0 & -1 & 1 & xq^3 &\ddots & & & & & & \\
0 & 0 & -1 &\ddots & \ddots & & & & & & \\
\vdots & \vdots & & \ddots & & & & & & &\vdots \\
   & & & & & & & & & \ddots& 0\\
  & & & & & & & & & & -x^{2k+1} q^{(2k+1)(N-k)+i}\\
   & & & & & & & &  & \ddots& 0\\
 & & & & & & & &  & \ddots& \vdots\\
  & & & & & & &\ddots & \ddots & \ddots& 0\\
 & & & & & & & \dots & 1 & x q^{N-2} & x^2 q^{2N-3} \\
\vdots & & & & & & & \dots& -1 & 1 & x q^{N-1}\\
0 & \dots & & & & & &\dots & 0 & -1 & 1\\
\end{array}\right|
\end{align*}
Expansion along the last column reveals that this expression satisfy Theorem~\ref{eq:F_Nijkx_rec}, and the expansion along the top row proves the assertion of this theorem.
\end{proof}

We remark that the recursion of Theorem~\ref{thm:F_N_ijkx_matrix_rec} is restricted to $j=0$ is because for $j>0$, the initial values of the determinant do not match the corresponding values of $F_N(i,j,k;x)$.

\begin{corollary}\label{cor:F_N_i0_rec}
\begin{equation}\label{eq:F_N_i0_rec} F_N(0,1,1;x) = (1+ xq)F_{N-1}(0,1,1;xq) + x^2 q^3 F_{N-2}(0,1,1,xq^2).\end{equation}
\end{corollary}

\begin{proof} We deduce from Corollary~\ref{cor:F_N011x_rec} that 
\[ F_{N}(0,1,1;x):= \left|\begin{array}{cccccccccc}
1+xq & -x^2q^3  &0 & \dots &  &    0 \\ 
-1 & 1+ x q^2 & -x^2q^5 & 0 & \dots & &    &     & \vdots \\ 
0 & -1 & 1+xq^3  & \ddots & &  &  &    &    \\ 
\vdots & \vdots & 0 & \ddots &    &    &  &  &    \\ 
 &  & \vdots & \ddots &    &   &  &  &   \\ 
 &  &  &  &  & &   &  & \vdots \\
  &  &  &  &  &  &       \ddots & \ddots & 0 \\ 
\vdots  &  &      &  & \ddots & 0 & -1 & 1+ xq^{N-1} & -x^2q^{2N-1} \\[-1ex]\\ 
0 & \dots &      &  & & \dots  & 0 & -1 & 1+xq^N
\end{array} \right|\]

Expanding along the last column reveals that this expression is indeed $F(0,1,1;x)$ by Corollary~\ref{cor:F_N011x_rec}. Expansion along the top row yields \eqref{eq:F_N_i0_rec}.
\end{proof}

\section{Reduction for $k=0$}\label{sec:F_k0_reduction}

\begin{theorem}\label{thm:F_N_k0_expression}
\begin{equation}
\label{eq:F_N_k0_expression}
F_N(i,j,0;x) = \sum_{n=0}^{j-1} {j-1\brack n }_q x^n q^{n(N+i+j)}.
\end{equation}
\end{theorem}

\begin{proof}
\begin{align*}
F_N(i,j,0;x) &= \sum_{m,n\geq 0}x^{m+n} (-1)^n q^{{n+1\choose 2}+m^2 +mn+i(m+n)} {N-n-m+j\brack m}_q {N-m\brack n}_q\intertext{Letting $M=n+m$ we rewrite the right side expression as follows.}
&= \sum_{M\geq0}x^M (-1)^M q^{{M+1\choose 2} + iM} \sum_{m\geq 0 } (-1)^m q^{m\choose 2} {N-M+j \brack m}_q {N-m \brack M-m}_q\\
\intertext{We can use the $q$-Chu--Vandermonde summation formula \cite[p.37, (3.3.10)]{Theory_of_Partitions} to simplify the inner sum and finish the proof.}
&=\sum_{M\geq0}x^M (-1)^M q^{{M+1\choose 2} + iM} {N\brack M}_q \frac{(-1)^M q^{NM-{M\choose 2}}(q;q)_{j-1}}{(q;q)_{j-M-1}}\\&=
\sum_{M\geq 0}x^M (-1)^M q^{M(N+i+1)} {j-1\brack M}_q.
\end{align*}
\end{proof}

\section{Proof of Theorem~\ref{thm:main_theorem}}\label{sec:Main_thm}

\begin{theorem}\label{thm:f_N+1} For non-negative integers $N,$ let \begin{equation}
\label{eq:fN_def}f_N(q) := \sum_{j\geq 0} q^{3j^2 -2j} {N \brack 3j}_q (q^2,q^3)_j,
\end{equation} then \begin{equation}\label{eq:lemma_f_N+1}f_{N+1}(q) = F_N(0,1,1;1) + q^N F_{N-1}(0,1,1;1).\end{equation}
\end{theorem}

\begin{proof} The $q$-Zeilberger algorithm (implemented in \cite{HolonomicFunctions} and many other places) is enough to prove that $f_N(q)$ satisfies the recurrence relation \begin{equation}\label{eq:rec_fN}
(1 - q^{N-2}) f_{ N}(q)- (1 - q^{ 2 N-3}) f_{ N-1}(q) +q^{2 N-4} (1 - q^{ N-1}) f_{N-2}(q) =0.
\end{equation}
Corollary~\ref{cor:F_N011x_rec} shows that
\[F_{N}(0,1,1;1) -(1-q^{N})F_{N-1}(0,1,1;1)+q^{2N-1}F_{N-2}(0,1,1;1)=0\] and similarly we can show that \[\hat{F}_{N}(0,1,1;1) -q(1-q^{N-1})\hat{F}_{N-1}(0,1,1;1)+q^{2N-1}\hat{F}_{N-2}(0,1,1;1)=0,\] where $\hat{F}_N(0,1,1;1) = q^N F_N(0,1,1;1)$.
We can find a recurrence satisfied by the sequence defined as the difference of $F_N(0,1,1;1)$ and $q^N F^{N-1}(0,1,1;1)$ using the closure properties of holonomic functions. This formal calculations are also included in the \texttt{qGeneratingFunctions} Mathematica package of Kauers and Koutschan \cite{qGeneratingFunctions}. Using the above recurrences we can then show that $b_{N+1} = F_N(0,1,1;1) - q^N F_{N-1}(0,1,1;1)$ satisfies the recurrence \begin{align}
\label{eq:rec_f_RHS}
 b_{N+1} -(1 + q) (1 + q^{N-1}) b_{N} +q (1 + q^{2 N-4} + &q^{N-1} + q^{N-1} + q^{2N-10})b_{N-1} q^{2N-3}\\\nonumber & - q^{2N-3} (1 + q) (1+ q^{NN-2}) b_{N-2} +q^{4N-8} b_{N-3}=0.  
\end{align}
At this stage, to prove the lemma, one can either find a recurrence satisfied by the difference $f_N(q)$ and $b_N$ using the same approach or checks to see if the recurrences \eqref{eq:rec_fN} and \eqref{eq:rec_f_RHS} have a common factor. The \texttt{qFunctions} package of the second author has the implementation of finding a greatest common factor of two given recurrences. Using that we show that $f_N(q)$ and $b_N=F_{N-1}(0,1,1;1) - q^{N-1} F_{N-2}(0,1,1;1)$ satisfies the same second order recurrence \eqref{eq:rec_fN}.

Now that we showed the left- and right-hand sides of \eqref{eq:lemma_f_N+1} satisfy the same recurrence \eqref{eq:rec_fN}. All we need to do is to show that the initial values match. For $N=0$ and $1$, it is easy to see that both sides of the claim gives 1. Therefore, since both sides satisfy the same second order recurrence with two equal initial conditions, the equation \eqref{eq:lemma_f_N+1} is true for all $N\geq 2$.
\end{proof}

Now we are finally equipped to prove Theorem~\ref{thm:main_theorem}. For $|q|<1$, taking the limit $N\rightarrow\infty$ of \eqref{eq:lemma_f_N+1} we get 
\begin{equation}\label{eq:limit_of_lemma}
\sum_{j\geq 0} \frac{(q^2;q^3)_\infty}{(q;q)_{3j}} q^{3j^2-2j} = \sum_{m,n\geq 0} \frac{(-1)^n q^{\frac{3n(3n+1)}{2} + m^2+ 3mn}}{(q)_m(q^3;q^3)_n}.
\end{equation}
The right-hand side of \eqref{eq:limit_of_lemma} is the left-hand side of \eqref{eq:1mod3eqn}. Moreover, the left-hand side series in \eqref{eq:limit_of_lemma} can be summed using the $q$-Gauss sum \cite[II.8, p.354]{Gasper_Rahman}, \begin{equation*}
\sum_{j\geq 0} \frac{(a,b;q)_j}{(q,c;q)_j} \left(\frac{c}{ab}\right)^j  = \frac{\left(\frac{c}{a},\frac{c}{b};q \right)_\infty}{\left(c,\frac{c}{ab};q \right)_\infty} ,
\end{equation*} with $(a,b,c,q)\mapsto(\rho,\rho,q,q^3)$ and later $\rho\rightarrow\infty$. This finishes the proof of Theorem~\ref{thm:main_theorem}.

\section{Overpartitions and Proof of Corollary~\ref{cor:Overpartition_corollary}}\label{sec:Overpartitions}

An overpartition of $n$ is an integer partition in which the first appearance of any summand may be overlined. For example, the eight overpartitions of 3 are \[3,\  \overline{3},\ 2+1,\ 2+\overline{1},\ \overline{2}+1,\ \overline{2}+\overline{1},\ 1+1+1,\ \overline{1}+1+1.\]

In the following we shall discuss sequences in overpartitions. We shall say that there is a sequence of form $a_1 + a_2 +\dots + a_r$ in the given overpartition if for some $j\geq 0$, $(a_1+j) + (a_2+j) + \dots +(a_r+j)$ appears as a subpartition of the given partition. For example, there is a sequence of the form $\overline{1}+ 2+ \overline{3}$ in $4+\overline{5}+6+\overline{7}+7$ because $\overline{(1+4)} + (2+4)+ \overline{(3+4)}$ appears in the given overpartition.

\begin{theorem}\label{thm:Overpartition_GF}
\[\frac{F(0,0,k;x)}{(xq;q)_\infty}\] is the generating function for the overpartitions, where the exponent of $x$ keeps track of the number of parts, in which \begin{enumerate}[i.]
\item $\overline{j}+\overline{(j+1)}$ does not appear,
\item there are no sequences of the form $\overline{1}+2+\overline{3}+4+\overline{5}+\dots+(2k)+\overline{(2k+1)}$.
\end{enumerate}
If $i>0$, \[\frac{F(i,0,k;x)}{(xq;q)_\infty}\] is the generating function for the overpartitions,in which
\begin{enumerate}[i.]
\setcounter{enumi}{2}
\item $\overline{j}+\overline{(j+1)}$ does not appear,
\item the smallest overlined part is $> i $,
\item sequences of the form \[2+3+\dots+i+\overline{(i+1)}+(i+1)+(i+2)+\overline{(i+3)} +(i+4) +\overline{(i+5)}+\dots +\overline{(2k)}+(2k+1)\] if $i$ is odd, and \[2+3+\dots+i+\overline{(i+1)}+(i+1)+(i+2)+\overline{(i+3)} +(i+4) +\overline{(i+5)}+\dots +(2k)+\overline{(2k+1)}\] if $i$ is even, are excluded.
\end{enumerate}
\end{theorem}

For example, with $k=i=1$, the coefficient of $q^7$ in $F(1,0,1;x)/(xq;q)_\infty$ is \[x^7 + 2x^6 + 4x^5 + 7x^4 + 10x^3 + 9x^2 + 2x.\] The ten indicated partitions of 7 with 3 parts are \[5+1+1,\ \overline{5}+1+1,\ 4+2+1,\ \overline{4}+2+1,\ 4+\overline{2}+1,\ \overline{4}+\overline{2}+1,\ 3+3+1,\ \overline{3}+3+1,\ 3+2+2,\ \overline{3}+3+2.\] Note that $\overline{3}+\overline{2}+2$, $5+\overline{1}+1$, $3+\overline{2}+2$ have been excluded by the conditions $iii.$, $iv.$ and $v.$, respectively.

\begin{proof} We take the limit as $N\rightarrow\infty$ in Theorem~\ref{thm:F_N_ijkx_matrix_rec}. Thus if \[f(i,j;x) := \frac{F(i,0,k;x)}{(xq;q)_\infty},\] then \begin{equation}
\label{eq:f_i0j_rec_for_overpartitions} 
f(i,k;x,q):=f(i,k;x) = \frac{1}{(1-xq)}f(i,k;xq) + \frac{xq^{i+1}}{(1-xq)(1-xq^2)}f(i,k;xq^2) - \frac{x^{2k+1}q^{{2k+2\choose 2}+i}}{(xq;q)_{2k+1}}f(i,k;xq^{2k+1}).
\end{equation}
Quite clearly $f(i,k;x)$ is uniquely defined by \eqref{eq:f_i0j_rec_for_overpartitions} give the boundary conditions $f(i,k;x,0) = f(i,k;0,q)=1$. Hence to prove our theorem, we need only to show that the generating functions for the overpartitions in question  satisfies the functional equation \eqref{eq:f_i0j_rec_for_overpartitions} plus the initial conditions. The boundary conditions are immediate from the fact that the empty partition of 0 is the only partition of non-positive number and the only partition with a nun-positive number of parts.

Let us now examine the three components of the right-hand side of \eqref{eq:f_i0j_rec_for_overpartitions}. The first term \[\frac{f(i,k;xq)}{(1-xq)}\] clearly accounts for those overpartitions in question that do not have $\overline{i+1}$ as a part.
The second term \[\frac{xq^{i+1}}{(1-xq)(1-xq^2)}f(i,k;xq^2)\] accounts for those overpartitions where now $\overline{(i+1)}$ appear.

This term has introduced disallowed partitions. Particularly, it has juxtaposed $\overline{(i+1)}$ with the subpartition $2+3+\dots+i+\overline{(i+1)}+(i+1)+(i+2)+\overline{(i+3)} +(i+4) +\overline{(i+5)}+\dots +\overline{(2k)+(2k+1)}$ (an overline appears either on $(2k)$ or on $(2k+1)$ as appropriate). This is not admissible.  

However, the term \[-\frac{x^{2k+1}q^{2+3+\dots+i+\overline{(i+1)}+(i+1)+(i+2)+\overline{(i+3)} +(i+4) +\overline{(i+5)}+\dots +\overline{(2k)+(2k+1)}}}{(xq;q)_{2k+1}}f(i,k;xq^{2k+1})\] removes the offending overpartitions. Thus the right-hand side of \eqref{eq:f_i0j_rec_for_overpartitions} accounts for precisely those overpartitions described by Theorem~\ref{thm:Overpartition_GF}.
\end{proof}

We now prove Corollary~\ref{cor:Overpartition_corollary}.
\begin{proof}[Proof of Corollary~\ref{cor:Overpartition_corollary}]
We recall Theorem~\ref{thm:main_theorem}, which asserts \[F(0,1;1) = \frac{1}{(q;q^3)_\infty}.\] Hence, \begin{equation}
\label{eq:F_011_divided_by_qPoch} \frac{F(0,1;1)}{(q;q)_\infty} = \frac{1}{(q;q)_\infty(q;q^3)_\infty}.
\end{equation}
The left-hand side of \eqref{eq:F_011_divided_by_qPoch} is the generating function for the overpartitions described in Corollary~\ref{cor:Overpartition_corollary} in the light of Theorem~\ref{thm:Overpartition_GF}. The right-hand side of the generating function for colored partitions describe in Corollary~\ref{cor:Overpartition_corollary}.
\end{proof}

\section{Some related continued fraction identities}\label{sec:ContFrac}

One can easily see that the three term relations of Corollaries \ref{cor:F_N011x_rec} and \ref{cor:F_N_i0_rec} give to finite continued fractions. We would like to start by recalling the initial conditions for $F_N(0,1,1;x)$ presented in the proof of Corollary~\ref{cor:F_N011x_rec}:
\[ F_0(0,1,1;x)=1\text{  and  }F_1(0,1,1;x)=1+xq.\]
By a simple rearrangement of terms in the said Corollaries \ref{cor:F_N011x_rec} and \ref{cor:F_N_i0_rec} and iteration of the formulas 
gives the two following results, respectively.

\begin{corollary} For $N\geq 1$, 
\begin{equation}\label{eq:ContFrac1}
\frac{F_N(0,1,1;x)}{F_{N-1}(0,1,1;x)} = 1+x q^N - \cfrac{x^2 q^{N-1}}{1+x q^{N-1} - \cfrac{x^2 q^{N-2}}{\ddots - \cfrac{\ddots}{1+x q^2 - \cfrac{x^2q}{1+xq}}}}.
\end{equation}
\end{corollary}

\begin{corollary}\label{cor:ContFrac2} For $N\geq 1$, 
\begin{equation}\label{eq:ContFrac2}
\frac{F_N(0,1,1;x)}{F_{N-1}(0,1,1;x q)} = 1+x q + \cfrac{x^2 q^{3}}{1+x q^{2} - \cfrac{x^2 q^{5}}{\ddots - \cfrac{\ddots}{1+x q^{N-1} + \cfrac{x^2q^{2N-1}}{1+xq^N}}}}.
\end{equation}
\end{corollary}

The continued fraction \eqref{eq:ContFrac1} tends to 1 as $N\rightarrow\infty$. On the other hand, as $N\rightarrow\infty$, we get the following result from Corollary~\ref{cor:ContFrac2} at $x=1$ by employing Theorem~\ref{thm:main_theorem} and simple manmipulations.

\begin{theorem}\label{thm:ContFracM}
\[(q;q^3)_\infty \sum_{m,n\geq 0} \frac{(-1)^n q^{\frac{3n(3n+1)}{2} + m^2+ 3mn +m +3n+1}}{(q)_m(q^3;q^3)_n} = \cfrac{q}{1+ q - \cfrac{ q^{3}}{1+ q^{2} - \cfrac{ q^{5}}{1+q^3 - \cfrac{q^7}{\ddots}}}}.\]
\end{theorem}

The right-hand side of Theorem~\ref{thm:ContFracM} is closely related to a continued fraction noted by Ramanujan \cite{GEA_CF1, GEA_CF2}: \begin{equation}\label{RamanujanCF}\frac{(q^2;q^3)_\infty}{(q;q^3)_\infty}=\cfrac{1}{1-\cfrac{q}{1+ q - \cfrac{ q^{3}}{1+ q^{2} - \cfrac{ q^{5}}{1+q^3 - \cfrac{q^7}{\ddots}}}}}.\end{equation} This yields the following theorem subject to some simple manipulations.

\begin{theorem}\label{thm:CF_thm}
\[\sum_{m,n\geq 0} \frac{(-1)^n q^{\frac{3n(3n+1)}{2} + m^2+ 3mn +m +3n+1}}{(q;q)_m(q^3;q^3)_n} = \frac{1}{(q;q^3)_\infty}-\frac{1}{(q^2;q^3)_\infty}.\] 
\end{theorem}

Moreover, using \eqref{eq:1mod3eqn} in Theorem~\ref{thm:CF_thm}, we get a series expansion for $1/(q^2;q^3)_\infty$ similar to \eqref{eq:1mod3eqn}.

\begin{theorem}
\[\sum_{m,n\geq 0} \frac{(-1)^n q^{\frac{3n(3n+1)}{2} + m^2+ 3mn }(1-q^{m +3n+1})}{(q;q)_m(q^3;q^3)_n} = \frac{1}{(q^2;q^3)_\infty}.\]
\end{theorem}


\section{Acknowledgment}

The first author would like to thank the Simons foundation for their support through the grant number 633284. The second author would like to thank UK Research and Innovation EPSRC and Austrian Science fund (FWF) for partially supporting his research through the grants EP/T015713/1 and P-34501N, respectively.

\end{document}